\documentclass[11pt,twoside,a4paper]{article}
\usepackage{fancyhdr}
\usepackage{makeidx}
\usepackage{multirow}
\usepackage{multicol}
\usepackage[dvipsnames,svgnames,table]{xcolor}
\usepackage{graphicx}
\usepackage{epstopdf}
\usepackage{ulem}
\usepackage{hyperref}
\usepackage{amsmath}
\usepackage{amsthm}
\usepackage{amssymb}

\newcommand{\eTRS}{\rm eTRS }

\newcommand{\TRS}{\rm TRS\ }

\newcommand{\SDPp}{\rm SDP}
\newcommand{\LNGM}{\rm LNGM\ }

\newtheorem{theorem}{Theorem}[section]

\newtheorem{lemma}{Lemma}[section]
\newtheorem{proposition}{Proposition}[section]
\newtheorem{corollary}{Corollary}[section]
\newtheorem{example}{Example}[section]
\title{On SOCP/SDP formulation of the extended trust region subproblem}
\author{S. Fallahi \footnote{Department of Mathematics, Salman Farsi University of Kazerun, Kazerun, Iran, Email: saeedfallahi808@gmail.com}
, M. Salahi\footnote{Corresponding Author.}$^,$ \footnote{Department of Applied Mathematics, Faculty of Mathematical Sciences, University of Guilan, Rasht, Iran, Email: salahim@guilan.ac.ir}
, S. Ansary  Karbasy\footnote{Department of Applied Mathematics, Faculty of Mathematical Sciences, University of Guilan, Rasht, Iran, Email:saeidansary144@yahoo.com}}
\date{}

\usepackage[paperwidth=612pt,paperheight=792pt,top=85pt,right=85pt,bottom=70pt,left=85pt]{geometry}

\begin{document}
\maketitle
\thispagestyle{fancy}
\fancyhead{}
%\lhead{\textit{Iranian Journal of Operations Research}\\Vol. x, No. x, xxxx, pp. xx-xx}
\renewcommand{\headrulewidth}{0pt}
\begin{abstract}
\noindent
In this paper, we consider the extended trust region subproblem
(\eTRS)  which is the minimization of an indefinite quadratic function subject to the intersection of unit ball  with a single linear
inequality constraint.  Using  a variation of S-Lemma, we derive the
necessary and sufficient optimality conditions for \eTRS.
Then  an SOCP/SDP formulation is introduced for the problem. Finally, several illustrative examples are provided.

% \PACS{PACS code1 \and PACS code2 \and more}
% \subclass{MSC code1 \and MSC code2 \and more}
\end{abstract}


{\scriptsize Manuscript was received on --/--/----, revised on --/--/---- and accepted for publication on --/--/----.}

\section{Introduction}\label{intro}

Consider the following { extended trust region subproblem \eTRS}
\begin{eqnarray}\min && x^{T} A x+2a^{T} x\nonumber\\
		   &&\|x\|^{2}\leq  1  \label{etrs1}\\
&& b^{T}x\leq \beta \nonumber
\end{eqnarray} where ${A}^{T} =A \in  \mathbb{R}^{n\times n}$ is indefinite, $a, b \in
\mathbb{R}^n$ and  $\beta\in \mathbb{R}.$ Since $A$ is indefinite,  it is a nonconvex optimization problem and semidefinite programming (SDP) relaxation is not tight  in general.
When $b=0$ and $\beta=0$, then \eTRS reduces to the well-known trust region subproblem (\TRS) which is the key subproblem in solving nonlinear optimization problems \cite{con}.
Although \TRS is a nonconvex problem, it enjoys strong duality and exact SDP relaxation exists for it \cite{fw}.
However, the following classical SDP relaxation is not exact for \eTRS as it will be shown also in the numerical results section:
\begin{align}\label{sdpr}
\min ~~ &A\bullet X + 2a^Tx \nonumber \\
            &  \mathrm{trace}(X)  \leq 1, \nonumber   \\
            & b^Tx\leq \beta,  \\
            &X \succeq xx^T. \nonumber
\end{align}

First the authors in \cite{sz} have studied \eTRS and proposed and exact SOCP/SDP\footnote{Second order cone program/Semidefinite program} formulation for it.
Due to the importance of \eTRS also in solving  general nonlinear optimization problems, several variants  of it have been the focus of  current research  \cite{bec,ba,by,hs,jl,sf}.
Beck and Eldar have studied \eTRS
under the condition that $\dim(Ker(A-\lambda_1 I))\geq 2$
which is equivalent to
\begin{equation}
	\label{eq:twosmal}
	 \lambda_1=\lambda_2,
\end{equation}
where $\lambda_1$ and $\lambda_2$ are the two smallest eigenvalues of $A$.
Under this condition, they have
shown that the following optimality conditions are necessary and sufficient
optimality conditions for \eTRS:
\[
	\begin{array}{cll}
		({\it i}) &2(A+\lambda I)x = -(2a + \mu b), \\
	({\it ii})&(A+ \lambda I) \succeq 0, \\
	({\it iii}) & \lambda(\|x\|^{2}- 1) = 0, \quad \mu (b^{T}x-
	\beta)=0,\\
	({\it iv}) & \lambda, \mu \geq 0.
	 \end{array}
\]
%Under condition (\ref{eq:twosmal}),  \TRS does not have  \LNGM.
Jeyakumar and Li in \cite{jl} have
shown that ${\rm dim ( \ Ker} (A-\lambda_1I_n))\ge 2$, together with the Slater
condition ensures that a set of combined first and second-order
Lagrange multiplier conditions are necessary and sufficient for the
global optimality of \eTRS and consequently for strong duality.
In \cite{hs} the authors have improved the dimension condition by
Jeyakumar and Li under which \eTRS admits an exact
semidefinite   relaxation.
They proposed the following condition
\index{\SDPp}
\begin{equation}
\label{eq:rankcond}
{\rm rank~} ([A- \lambda_1I_{n}~~ b ] ) \leq n-1.
\end{equation}

It should be noted that \TRS has at most one local-nonglobal minimum (\LNGM) \cite{mart}, which is a candidate for the optimal solution of \eTRS if it is feasible. An efficient algorithm for computing \LNGM is given in \cite{stw}. All the above  rank conditions guarantee that the global solution of \eTRS does not happen at the \LNGM of \TRS. Most recently  in \cite{ba}  the authors have studied \eTRS and derived the SOCP/SDP reformulation given in \cite{sz} by different approach and extended it the cases where more than one linear constraint exist. In this paper, using a variation of S-Lemma, first we derive the necessary and sufficient optimality conditions for \eTRS which leads us to an SOCP/SDP formulation of it. Then we prove that our derived formulation   is the dual of the formulation presented in \cite{ba,sz}.  Finally, we present several numerical examples illustrating various cases that may happen for the optimal solution of \eTRS.
\section{Global Optimality Conditions for \eTRS}
\label{sect:newglobal}
%In this section, we present a necessary and sufficient condition for
%global optimality of \eTRS under the Slater's condition and without any further assumptions.
We define the {dual cone of $S$} as $S^*= \{y: \langle
y,x\rangle \geq 0, \forall x \in S\}$. The following proposition which is variant of S-Lemma, plays a key role  in the proof of the
optimality conditions.
\begin{proposition}
	\label{prop:Slemma}
 Let $f, g : \mathbb{R}^n \longrightarrow  \mathbb{R}$  be
quadratic functions, $g(x)=x^{T}A_{g}x + {a_{g}}^{T}x + c_{g}$, and let
$b \in \mathbb{R}^n$, and
$\beta \in \mathbb{R}$. Moreover, assume that $g(x)$ is convex and there
exists an $\bar{x} \in \mathbb{R}^n$ such that $b^{T}\bar{x} < \beta$
and $g(\bar{x}) < 0$. Then the  following two statements are equivalent:
\begin{enumerate}
\item The system
 \begin{eqnarray}&& f(x) < 0,\nonumber\\
&&g(x) \leq 0,\nonumber\\
&& b^{T}x \leq \beta ,\nonumber\\
&& ~~x \in \mathbb{R}^n \nonumber
\end{eqnarray}
is not solvable.
\item  There is a nonnegative multiplier $y \geq  0$, a scalar
	$u_0\in \mathbb{R}$ and vector $u \in \mathbb{R}^n$ such that
\[
	\begin{array}{c}
 f(x) + y g(x) + (u^{T}x-u_{0})(b^{T}x - \beta)\geq  0, ~~ \forall x \in \mathbb{R}^n,\\
 u\in  \{ x\in \mathbb{R}^n:  x^{T}A_{g}x \leq 0, {a_{g}}^{T}x \leq 0\}^*,\\
 \begin{pmatrix}
	 u_0 \cr u
 \end{pmatrix}
 \in  \left\{
 \begin{pmatrix}
	 x_0 \cr x
 \end{pmatrix} :   x_0=-1,
	  g(x) \leq 0, c_{g}+{a_{g}}^{T}x \leq 0 \right\}^*.
\end{array}
\]
\end{enumerate}
\end{proposition}
\begin{proof}
	See \cite[Corollary 7]{sz}.
\end{proof}

\begin{corollary}
When $g(x)=\|x\|^2-1$, then the second item in  Proposition 1 is equivalent to
\begin{eqnarray}&& f(x) + y g(x) + (u^{T}x-u_{0})(b^{T}x - \beta)\geq  0, ~~ \forall x \in \mathbb{R}^n,\nonumber\\
 &&\begin{pmatrix}
	 -u_0 \cr u
 \end{pmatrix}
 \in  L_{n+1},\nonumber
\end{eqnarray}
where $L_{n+1}$ is the {Lorentz cone} defined as follows:
\[L_{n+1}=\{x=(x_0;\bar x)\in \mathbb{R}^{n+1}|\ \ \|\bar x\|\le x_0\}.\]
\end{corollary}
In the following theorem, we give the  optimality conditions for
\eTRS.  Our proofs follows the idea in \cite{jl}.
\begin{theorem}\label{thm1}
	Suppose that the { strict feasibility constraint
	} holds for \eTRS:
\[
	\exists \hat{x} \in \mathbb{R}^n  \text{  with  }
	\|\hat{x}\|^{2}- 1<0, \ b^{T}\hat{x} -\beta <0.
\]
Moreover, let $x^{*}$ be a feasible point for \eTRS. Then $x^{*}$ is a global
minimizer of \eTRS if, and only if, there exist $\lambda_{0} \in \mathbb{R}_{+}$ and $(-u_{0}, u) \in L^{n+1}$ such that the following conditions hold:\\
\begin{itemize}
	\item (i)  $(2A+2\lambda_{0}I + bu^{T}+ ub^{T})x^{*} = -(2a - \beta u -bu_{0}),$ %\hspace{10mm} (KKT Condition)\\
	\item (ii)  $\lambda_{0}(\|x^{*}\|^{2}- 1) = 0, (u^{T}x^{*}- u_{0})(b^{T}x^{*}- \beta)=0,$ % (Complementary Slackness) \\
	\item (iii) $(2A+2\lambda_{0}I + bu^{T}+ ub^{T}) \succeq 0$. %\qquad \hspace{17mm}(Second Order Condition).\\
%and $u^{T}x -u_{0}\geq 0, ~~ \forall x \in \mathbb{R}^n$ such that $ \| x \|^{2} \leq \delta.$
\end{itemize}
\end{theorem}
\begin{proof}
	$[${\bf Necessity}$]$ Let $x^{*}$ be a global minimizer of
	\eTRS. Then the
following system of inequalities has no solution:
 \begin{eqnarray}&& x^{T} A x+{2a}^{T} x+\gamma <0,\nonumber\\
&&\hspace{9mm} \| x \|^{2}- 1 \leq 0,\nonumber\\
&& \hspace{16mm} b^{T}x \leq \beta,\nonumber
\end{eqnarray}
where $\gamma=-({x^{*}}^{T} A x^{*}+{2a}^{T} x^{*})$.
Thus by Proposition \ref{prop:Slemma} there exist $\lambda_{0} \geq  0$ and a vector $(u_0; u) \in \mathbb{R}^{n+1}$ such
that
\[x^{T} A x+{2a}^{T} x +\gamma + \lambda_{0}(\| x \|^{2}- 1) + (u^{T}x-u_{0})(b^{T}x - \beta)\geq  0, ~~ \forall x \in \mathbb{R}^n\]
and $u^{T}x -u_{0}\geq 0, ~~ \forall x \in \mathbb{R}^n: \| x \|^{2} \leq \delta^2.$ Let $x=x^{*}$, then we have
\[\lambda_{0}(\| x^{*} \|^{2}- 1) + (u^{T}x^{*}-u_{0})(b^{T}x^{*} - \beta)\geq  0.\]
Now as $x^{*}$ is feasible for \eTRS and $\lambda_{0}\geq 0,\ (u^{T}x^{*}-u_{0})\geq 0$, it follows that
\[\lambda_{0}(\| x^{*} \|^{2}- 1)=0, \qquad (u^{T}x^{*}-u_{0})(b^{T}x^{*} - \beta)=  0.\]
Let
\[h(x)=x^{T} A x+{2a}^{T} x  + \lambda_{0}(\| x \|^{2}- 1) + (u^{T}x-u_{0})(b^{T}x - \beta),\]
then obviously $x^{*}$ is a global minimizer of $h$, and so $\nabla h(x^{*})=0$ and $\nabla^{2} h(x^{*})\succeq 0$ i.e.,
\[(2A+2\lambda_{0}I + bu^{T}+ ub^{T})x^{*} = -(2a - \beta u -bu_{0}),\]
\[ (2A+2\lambda_{0}I + bu^{T}+ ub^{T}) \succeq 0.\]
Thus all conditions (i), (ii) and (iii) hold.\\
$[${\bf Sufficiency}$]$ If the optimality conditions hold, then from (ii) we see that
\[h(x)=x^{T} A x+{2a}^{T} x  + \lambda_{0}(\| x \|^{2}- 1) + (u^{T}x-u_{0})(b^{T}x - \beta)\]
is convex. Moreover, from condition (i), we have  $\nabla h(x^{*})=0$,
therefore, $x^{*}$ is a global minimizer of $h$. Thus for given
$\lambda_0$ and $(u_0;u)$ in the optimality conditions and for any
feasible solution of \eTRS, we have
 \begin{eqnarray} x^{T} A x+{2a}^{T} x \hspace{-.5cm}&& \geq x^{T} A x+{2a}^{T} x + \lambda_{0}(\| x \|^{2}- 1) + (u^{T}x-u_{0})(b^{T}x - \beta),\nonumber\\
&&\geq {x^{*}}^{T} A x^{*}+{2a}^{T} x^{*} + \lambda_{0}(\| x^{*} \|^{2}- 1) + (u^{T}x^{*}-u_{0})(b^{T}x^{*} - \beta)\nonumber\\
&& = {x^{*}}^{T} A x^{*}+{2a}^{T} x^{*}.\nonumber
\end{eqnarray}
This implies that $x^{*}$ is a global minimizer of \eTRS.
\end{proof}
\begin{theorem}
Suppose that there exists $\bar{x} \in \mathbb{R}^n$  with $\|\bar{x}\|^{2}- 1<0$ and $b^{T}\bar{x} -\beta<0$. Then,
we have
 \begin{eqnarray}\label{dual}
&&\min \Big \lbrace x^{T} A x+{2a}^{T} x : \| x \|^{2}\leq 1,
	 b^{T}x \leq \beta \Big \rbrace \nonumber \\&&=
	 \max_{\lambda_{0}\geq 0, (u_{0}, u) \in S } \min_{x} \Big
	 \lbrace x^{T} A x+{2a}^{T} x + \lambda_{0}(\| x \|^{2}- 1) + (u^{T}x-u_{0})(b^{T}x - \beta) \Big \rbrace,  \nonumber
\end{eqnarray}
where
$$ S=\Big \lbrace (u_{0};u) \Big |  (u^{T}x-u_{0})\geq 0,\quad  \forall x: \| x \|^{2} \leq 1 \Big \rbrace $$
and the maximum is attained.
\end{theorem}
\begin{proof}
%\[L(x,\lambda_{0}, u,u_{0})=x^{T} A x+{b}^{T} x  + \lambda_{0}(\| x \|^{2}- 1) + (u^{T}x-u_{0})(a^{T}x - \alpha),\]
It is easy to see that, for every feasible point of \eTRS, and every $\lambda_{0}\geq 0$ and $(u_0,u) \in S$,
\[x^{T} A x+{2a}^{T} x \geq x^{T} A x+{2a}^{T} x  + \lambda_{0}(\| x \|^{2}- 1) + (u^{T}x-u_{0})(b^{T}x - \beta).\]
Therefore
 \begin{eqnarray}
&&\min \Big \lbrace x^{T} A x+{2a}^{T} x : \| x \|^{2}\leq 1,
	 b^{T}x \leq \beta \Big \rbrace \nonumber \\&&\geq
	 \max_{\lambda_{0}\geq 0, (u_{0}, u) \in S} \min_{x} \Big
	 \lbrace x^{T} A x+{2a}^{T} x + \lambda_{0}(\| x \|^{2}- 1) + (u^{T}x-u_{0})(b^{T}x - \beta) \Big \rbrace. \nonumber
\end{eqnarray}
To show the reverse inequality, let $x^{*}$ be a global minimizer of
\eTRS, then there exists $\lambda_{0} \in \mathbb{R}_{+}$ and $(u_{0}, u) \in S$ such that the following condition holds:
\begin{itemize}
\item $(2A+2\lambda_{0}I + bu^{T}+ ub^{T})x^{*} = -(2a - \beta u -bu_{0}),$
\item  $\lambda_{0}(\|x^{*}\|^{2}- 1) = 0$ and $ (u^{T}x^{*}-u_{0})(b^{T}x^{*}- \beta)=0,$
\item $(2A+2\lambda_{0}I + bu^{T}+ ub^{T}) \succeq 0.$
\end{itemize}
We see that
\[h(x)=x^{T} A x+{2a}^{T} x  + \lambda_{0}(\| x \|^{2}- 1) + (u^{T}x-u_{0})(b^{T}x - \beta),\]
is convex, $\nabla h(x^{*})=0$ and $\nabla^{2} h(x^{*})\succeq 0$. Therefore, $x^{*}$ is a global minimizer of $h$ i.e.,
 \begin{eqnarray}  &&  x^{T} A x+{2a}^{T} x + \lambda_{0}(\| x \|^{2}-
\delta^2) + (u^{T}x-u_{0})(b^{T}x - \beta)\nonumber \\&& \geq {x^{*}}^{T}
Ax^{*}+{2a}^{T} x^{*} + \lambda_{0}(\| x^{*} \|^{2}- 1) + (u^{T}x^{*}-u_{0})(b^{T}x^{*} - \beta)\nonumber\\
&& = {x^{*}}^{T} A x^{*}+{2a}^{T} x^{*}.\nonumber
\end{eqnarray}
Therefore
\begin{eqnarray}
&&\min \Big \lbrace x^{T} A x+{2a}^{T} x : \| x \|^{2}\leq 1,
	b^{T}x \leq \beta \Big \rbrace \nonumber \\&&\leq
	\max_{\lambda_{0}\geq 0,\ (u_{0}; u) \in S} \min_{x} \Big \lbrace
	x^{T} A x+{2a}^{T} x + \lambda_{0}(\| x \|^{2}- 1) + (u^{T}x-u_{0})(b^{T}x - \beta) \Big \rbrace. \nonumber
\end{eqnarray}\end{proof}
As we see, in general  strong duality does not hold for \eTRS  which is the reason to rank conditions given  in \cite{bec,hs,jl} to guarantee it.
\begin{corollary}
If $u=0$ and $u_0\neq 0$, then strong duality holds for \eTRS.
\end{corollary}
\begin{proof}
It follows from the previous theorem.
\end{proof}

Form Theorem 2.2, we further have
\begin{eqnarray}\label{main}&&\max_{\lambda_{0}\geq 0, (u_{0}; u) \in S}
\min_{x} \Big \lbrace x^{T} A x+{2a}^{T} x + \lambda_{0}(\| x \|^{2}- 1) + (u^{T}x-u_{0})(b^{T}x - \beta) \Big \rbrace\nonumber\\ &&  =\max z \nonumber\\
&& \left( \begin{array}{cc}  -\lambda_0 + \beta u_0 - z  &\ \  \displaystyle \frac{1}{2}\left( 2a-\beta u - bu_0 \right)^T  \\ \displaystyle  \frac{1}{2}\left( 2a-\beta u - bu_0 \right) &~~ A+\lambda_0  I + \frac{1}{2}\left( bu^T + ub^T \right)  \end{array}\right) \succeq 0,\\
&& \Vert u \Vert \leq -u_{0},\nonumber\\
&& \lambda_{0}\geq 0,\nonumber
\end{eqnarray}
which is an SOCP/SDP formulation for \eTRS. In what follows, we show that this formulation is the dual of the SOCP/SDP formulation given  in \cite{ba,sz}.
Consider the Lagrange function of (\ref{main}):
\begin{align*}
\mathcal{L}(Y,v,u,u_0,\lambda_0 ,z)&=z +
\begin{pmatrix}
-\lambda_0 + \beta u_0 - z & \frac{1}{2}\left( 2a-\beta u - bu_0 \right)^T \\
\frac{1}{2}\left( 2a-\beta u - bu_0 \right) & A+\lambda_0  I + \frac{1}{2}\left( bu^T + ub^T \right)
\end{pmatrix}\bullet Y +v^T
\begin{pmatrix}-u_0 \\ u \end{pmatrix}
\end{align*}
where  $Y\succeq0$ and $||\bar{v}|| \leq v_0 $.
Let also
\begin{align*}
Y=\begin{pmatrix}
\alpha & x^T \\
x &X
\end{pmatrix}.
\end{align*}
 Thus the Lagrangian can be written as
\begin{align*}
\mathcal{L}(Y,v,u,u_0,\lambda_0 ,z)&=z +
\left(A+\lambda_0  I + \frac{1}{2}\left( bu^T + ub^T \right) \right) \bullet X +
\left(  2a-\beta u - bu_0 \right)^Tx \\
& ~~~~~~+\alpha \left(-\lambda_0 + \beta u_0 - z\right)
+\bar{v}^Tu - v_0 u_0 \\
&=A\bullet X + 2a^Tx +(1 - \alpha)z +\lambda_0 \left(\mathrm{trace}(X) - \alpha \right) + (Xb   -\beta x+ \bar{v})^Tu\\
&~~~~~~~~~~~~+(-b^Tx - v_0 + \beta ) u_0.
\end{align*}
Therefore, the Lagrangian dual becomes
\begin{align*}
&\operatornamewithlimits{\min}_{\begin{pmatrix}\alpha & x^T \\ x & X \end{pmatrix} \succeq 0,~||\bar{v}|| \leq v_0} ~~ \max_{\lambda_0 \geq 0,~ ||u|| \leq -u_0 } ~~\mathcal{L}(Y,v,u,u_0,\lambda_0 ,z)\\
&=\operatornamewithlimits{\min}_{\begin{pmatrix}\alpha & x^T \\ x & X \end{pmatrix} \succeq 0,~||\bar{v}|| \leq v_0} ~~ \max_{\lambda_0 \geq 0,~ ||u|| \leq -u_0 }A\bullet X + 2a^Tx + (1 - \alpha)z +\lambda_0 \left(\mathrm{trace}(X) - \alpha \right) \\
&\hspace{7.4cm} + (Xb   -\beta x  + \bar{v})^Tu+(-b^Tx - v_0 + \beta ) u_0 \\
&=\operatornamewithlimits{\min}_{\begin{pmatrix}\alpha & x^T \\ x & X \end{pmatrix} \succeq 0,~||\bar{v}|| \leq v_0} ~~\mathcal{G}(X,x,\alpha),
\end{align*}
where%%%%%%%%%%%%%%%%%
\begin{align*}
\hspace{-2cm}&\mathcal{G}(X,x,\alpha)= \max_{\lambda_0 \geq 0,~ ||u|| \leq -u_0 }A\bullet X + 2a^Tx + (1 - \alpha)z +\lambda_0 \left( \mathrm{trace}(X) - \alpha \right)\\
&\hspace{4.9cm}+ (X b -\beta x + \bar{v})^Tu+(-b^Tx - v_0 +\alpha \beta ) u_0.
\end{align*}
We further have
\begin{align*}
\mathcal{G}(X,x,\alpha) =
\begin{cases}
A\bullet X + 2a^Tx, & \mbox{if }~1-\alpha=0 ,~ \mathrm{trace}(X) - \alpha \leq 0,~  \\
                         &~~~~Xb -\beta x   + \bar{v} =0,~-b^Tx - v_0 + \beta\geq 0\\
\infty, & \mbox{o.w}
\end{cases}
\end{align*}
Thus Lagrangian dual becomes
\begin{align}
\min ~~ &A\bullet X + 2a^Tx\nonumber \\
            &  \mathrm{trace}(X) - 1\leq 0,\nonumber \\
            &Xb -\beta x + \bar{v} =0,\nonumber \\
            &-b^Tx - v_0 + \beta\geq 0,\label{f2}\\
            &||\bar{v}|| \leq v_0,\nonumber \\
            &X \succeq xx^T.\nonumber
\end{align}
From (\ref{f2}) we have:
\begin{align*}
&\bar{v}=  \beta x - Xb,  \\
&v_0 \leq -b^Tx   + \beta, \\
&||\bar{v}|| \leq v_0 \Longrightarrow ||\bar{v}|| = || \beta x- Xb   || \leq v_0 \leq -b^Tx   + \beta.
\end{align*}
Therefore (\ref{f2}) can be written as follows:
\begin{align}\label{dsocp}
\min ~~ &A\bullet X + 2a^Tx, \nonumber \\
             &  \mathrm{trace}(X)  \leq 1, \nonumber   \\
             &|| \beta x-Xb  || \leq -b^Tx   + \beta,  \\
             &X \succeq xx^T. \nonumber
\end{align}
This SOCP/SDP formulation is exactly the one given in \cite{ba, sz} but our derivation is completely different.
\begin{corollary}
If at the optimal solution of (\ref{dsocp}), $X^*=x_{{\tiny socp/sdp}}^*(x_{{\tiny socp/sdp}}^*)^T$, then $x^*_{{\tiny socp/sdp}}$ is optimal for (\ref{etrs1}).
\end{corollary}
\subsection{Rank one decomposition  procedure}
In order to derive an optimal solution for \eTRS from the none-rank one solution of (\ref{dsocp}), here we give a rank one decomposition approach similar to the one in  \cite{yz}.  Let $X^*$ be an optimal solution for (\ref{dsocp}) which is not rank one and consider the following notations:
\begin{align*}
&Y^*=\begin{pmatrix}
1  ~&\left( x_{{\tiny socp/sdp}}^* \right)^T\\
x_{{\tiny socp/sdp}}^*  ~& X^*  \\
\end{pmatrix}\hspace{-0.15cm},~
J=\begin{pmatrix}
1  ~& 0\\
0  ~& -I_n \\
\end{pmatrix}\hspace{-0.15cm},~
g=\begin{pmatrix}
\beta  \\
-b   \\
\end{pmatrix}\hspace{-0.13cm}.
\end{align*}
Obviously we have
\begin{align*}
& || \beta x_{{\tiny socp/sdp}}^*-X^*b  || \leq -b^Tx_{{\tiny socp/sdp}}^*   + \beta ~~ \Longleftrightarrow ~
\begin{pmatrix}
\beta -b^Tx_{{\tiny socp/sdp}}^*   \\
\beta x_{{\tiny socp/sdp}}^*-X^*b   \\
\end{pmatrix} = Y^*g \in L_{n+1}, \\
&\mathrm{trace}(X^*) \leq 1 ~~ \Longleftrightarrow ~ J\bullet Y^* \geq 0.
\end{align*}
\begin{lemma}[\cite{yz}]\label{teo22}
Let $G$ be an arbitrary symmetric matrix and $X$ be a positive semidefinite matrix with rank $r$. Further suppose that $G\bullet X\geq 0$. Then there exists a rank-one decomposition for $X$ such that
$$ X=\sum_{i=1}^r x_i x_i^T     $$
and $x_i^T G x_i \geq 0$  for all $ i=1,\cdots, r$. If, in particular,  $G\bullet X= 0$, then  $x_i^T G x_i = 0$ for all  $ i=1,\cdots, r$.
\end{lemma}
The following case may occur:

{\bf Case 1.} $Y^*g = 0$. From Lemma \ref{teo22}, there exists a rank one decomposition
for $Y^*$ as follows:
\begin{align*}
Y^* = \sum_{i=1}^r (y_i^*)(y_i^*)^T,
\end{align*}
where $r$ is the rank of $Y^*$ such  $J\bullet \left[ (y_i^*)(y_i^*)^T \right]\geq 0$ for all $i=1, \ldots, r$. Morover,
$J\bullet \left[ (y_i^*)(y_i^*)^T \right]= 0$ for all $i=1, \ldots, r$  if  $J\bullet Y^*=0$. We may choose the sign of the $y_i^*$
to ensure that $y_i^* \in L_{n+1}$, $i=1, \ldots, r$.

By linear independence of $y_i^*$'s, we get $g^T y_i^* =0$, $i=1, \ldots, r$. Let
$y_i^*=\begin{pmatrix} t_i^*   \\ \bar{y}_i^*  \\ \end{pmatrix}$, $i=1, \ldots, r$. Since $y_i^* \in L_{n+1}$ and  $y_i^*\neq 0$, we have $t_i^* >0$, $i=1, \ldots, r$. Take any $1\leq j\leq r$; it follows that
$\begin{pmatrix}  1  \\ \bar{y}_i^*/ t_i^* \\ \end{pmatrix} \left(1 ~~ [\bar{y}_i^*/ t_i^*]^T   \right)$ is optimal for (\ref{dsocp}).

{\bf Case 2.}  $J\bullet Y^*>0$  and  $Y^*g\neq 0$. Due to the complementarity condition, we must have $\lambda_0 =0$. Let
$y^*_g :=Y^*g= \begin{pmatrix}  t_g^*  \\ \bar{y}_g^* \\ \end{pmatrix}$. Since $y_g^* \in L_{n+1}$ by feasibility, we know that $t_g^*>0$. Moreover, $J\bullet \left[y_g^* (y_g^*)^T  \right] = (t_g^*)^2 - || \bar{y}_g^* ||^2\geq 0$ , and
$y_g^* (y_g^*)^Tg = \left( g^T Y^* a \right)Y^* a \in L_{n+1}$. Therefore, $y_g^* (y_g^*)^T / (t^*_g)^2$ is optimal for (\ref{dsocp})  as it is feasible and satisfies the complementarity conditions.

{\bf Case 3.}  $J\bullet Y^*=0$  and  $Y^*g\neq 0$. Denote  $y^*_g :=Y^*g\neq 0$. Let $\tilde{Y}= Y^* - \frac{Y^* gg^TY^*}{g^TY^*g}\succeq 0$. It is easy to see that $\tilde{Y}g=0$.

Case 3.1. $J \bullet \left[y^*_g (y^*_g)^T \right] = 0$. In this subcase, we have that  $y^*_g (y^*_g)^T/(t^*_g)^2$ is optimal for (\ref{dsocp}).

Case 3.2. $J \bullet \left[y^*_g (y^*_g)^T \right] >0 $. In this subcase,
\begin{align}\label{fr28}
J \bullet \tilde{Y} = J \bullet Y^* - J \bullet \left[y^*_g (y^*_g)^T \right]/g^TY^* g <0.
\end{align}
Now let us decompose  $ \tilde{Y}$ as
\begin{align*}
 \tilde{Y} = \sum_{i=1}^s \tilde{y}_i \tilde{y}_i^T,
\end{align*}
where $s= \mathrm{rank}(\tilde{Y}) >0$.  Since $\tilde{Y}g=0$, we have $ \tilde{y}_i^Tg=0$,
for all $i = 1, \ldots , s$. Choose $j$ such that
\begin{align*}
 J \bullet \tilde{y}_j \left(\tilde{y}_j \right)^T <0.
\end{align*}
Such $j$ must exist due to (\ref{fr28}). Consider the following quadratic equation:
\begin{align*}
 J \bullet \left[\left( y^*_g  +\alpha \tilde{y}_j \right) \left( y^*_g  +\alpha \tilde{y}_j \right)^T \right]=0.
\end{align*}
This equation has two distinct real roots with opposite signs.
Choose  the one such that the first component of $y^*_g  +\alpha \tilde{y}_j$ is positive.
 Denote
\begin{align*}
y^*_g +\alpha \tilde{y}_j  :=\begin{pmatrix}  t^*  \\ \bar{y}^* \\ \end{pmatrix}\hspace{-0.12cm}.
\end{align*}
In this case, since $ J \bullet \left[ y^*_g \left( y^*_g  \right)^T \right]>0$ it follows that $y^*_g$   is in the strict interior of the
cone $L_{n+1}$. Due to the complementarity, we must have $( u_0^*; u^*) =0$. Let us consider the solution
$ \begin{pmatrix} 1  \\ \bar{y}^* /t^* \\ \end{pmatrix} \left(    1   ~~ (\bar{y}^* /t^*)^T \right)$. It is easy to check that  this solution is both
feasible and complementary to the dual optimal solution  $ \left( \lambda_0^*; u_0^*; u^*\right)$, thus optimal for (\ref{dsocp}).

\section{Numerical examples}
The aim of this section is to provide various examples explaining different cases that might occur for the optimal solution of \eTRS.
\hspace{-6mm}
%%%%%%%%%%
\begin{example}
Consider the following example:
\begin{align*}
A=\begin{pmatrix}
-4 ~& 0 ~& 0\\
0  ~& 12 ~& 0\\
 0 ~& 0 ~&11\\
\end{pmatrix}\hspace{-0.15cm},~
a=\begin{pmatrix}
-4  \\
0   \\
 0  \\
\end{pmatrix}\hspace{-0.15cm},~
b=\begin{pmatrix}
20  \\
8   \\
 -14  \\
\end{pmatrix}\hspace{-0.15cm},~
\delta =1,~\beta = 5.
\end{align*}
We have $\lambda_1 = -4$ and
$\mathrm{dim} ~ \mathrm{Ker}\left( A- \lambda_{min}(A) I_n\right)=1 \ngeq  2$, thus the dimension condition (\ref{eq:rankcond})   fails to hold. Moreover, the new dimension condition given in \cite{hs} also fails to hold, since
\begin{align*}
\mathrm{rank}\left( \left[ A-\lambda_1 I_n ~~~  b \right] \right)=
\begin{pmatrix}
0  ~& 0 ~& 0 &20 \\
0  ~& 8 ~& 0 &8\\
 0 ~& 0 ~& 7 &-14\\
\end{pmatrix}=3 \nleqslant 2.
\end{align*}
The optimal objective value of SDP relaxation (\ref{sdpr}) is $-7.6827$. The global solution of \TRS is $x_g^*=[1, 0, 0]^T$ which is infeasible for \eTRS and \LNGM of \TRS is ${x}_l^*=[ -1~, 0~,0]^T$ which is feasible for \eTRS with the objective value $4.0000$. Moreover, for (\ref{dsocp}), the optimal solution is
$x_{{\tiny socp/sdp}}^*=[0.6266 , -0.2169 ,   0.4140]^T$ and $X^*=x_{{\tiny socp/sdp}}^*(x_{{\tiny socp/sdp}}^*)^T$. Thus $x_{{\tiny socp/sdp}}^*$ is optimal for (\ref{etrs1}) with objective value  $-4.1329$. As we see, the classical SDP relaxation (\ref{sdpr}) is not exact for this example and subsequently strong duality fails to hold. Also it is worth to note that at the optimal solution,  the linear constraint is active while the trust region constraint is not active.

%\\
%\\
%However, $x^*=[0.6266 , -0.2169 ,   0.4140]^T$, $\lambda_ 0= 0$, and
%$(u_0, u) = ( 0.2184 ~,   -0.0426~,  0.0622  ~,  -0.4788 ) $ satisfy the optimality conditions given in Theorem 1.  Since $A+\lambda_0I_n\nsucceq 0$, then strong duality  fails to hold. In this the example linear inequality constraints is active and ellipsoid constraint is inactive. LNGM for the example is $\bar{x}=[ -1~, 0~,0]$.
\end{example}
%%%%%%%%%%%%%%%%%%%%%%%%%%%%%%%
%%%%%%%%%%%%%%%%%%%%%%%%%%%%%%%%%%%%%%
\begin{example}
Consider the following example:
\begin{align*}
A=\begin{pmatrix}
-4 ~& 0 ~& 0\\
0  ~& 5 ~& 0\\
 0 ~& 0 ~&3\\
\end{pmatrix}\hspace{-0.15cm},~
a=\begin{pmatrix}
0.5714  \\
0   \\
 0  \\
\end{pmatrix}\hspace{-0.15cm},~
b=\begin{pmatrix}
-17  \\
14   \\
 -2  \\
\end{pmatrix}\hspace{-0.15cm},~
\delta =1,~\beta = 4.4.
\end{align*}
We have $\lambda_1 = -4$ and
$\mathrm{dim} ~ \mathrm{Ker}\left( A- \lambda_{min}(A) I_n\right)=1 \ngeq  2$, thus the dimension condition (\ref{eq:rankcond})   fails to hold for this example as well. Also the new dimension condition \cite{hs} fails to hold here, since
\begin{align*}
\mathrm{rank}\left( \left[ A-\lambda_1 I_n ~~~  b \right] \right)=
\begin{pmatrix}
0  ~& 0 ~& 0 &-17 \\
0  ~& 5 ~& 0 &14\\
 0 ~& 0 ~& 3 &-2\\
\end{pmatrix}=3 \nleqslant 2.
\end{align*}
The global solution of \TRS is $x_g^*=[-1, 0, 0]^T$ which is infeasible for \eTRS and \LNGM of \TRS is $x_l^*=[ 1~, 0~,0]^T$ which is feasible for \eTRS with the objective value $-2.4972$. The optimal objective value of SDP relaxation (\ref{sdpr}) is $-5.4326$ and the optimal objective value of  SOCP/SDP formulation (\ref{dsocp}) is $-2.4972$ which is also the optimal objective value of (\ref{etrs1}). Moreover, for (\ref{dsocp}), the optimal solution is
${x_{{\tiny socp/sdp}}}^*=[ 1~, 0~,0]^T$ and $X^*=x_{{\tiny socp/sdp}}^*(x_{{\tiny socp/sdp}}^*)^T$ and thus $x_{{\tiny socp/sdp}}^*$ is optimal for (\ref{etrs1}). Here also strong duality fails to hold like the previous example. Finally, at the optimal solution,  the linear constraint is not active while the trust region constraint is   active.
\end{example}
%%%%%%%%%%%%%%%%%%%%%%%%%%%%%%%%%%%%%%%%
%%%%%%%%%%%%%%%%%%%%%%%%%%%%%%%%%%%%%%%%%%%
\begin{example}\label{Exam3}
Consider the following example where at optimality both constraints are active:
\begin{align*}
A=\begin{pmatrix}
-4 ~& 0 ~& 0\\
0  ~& -8 ~& 0\\
 0 ~& 0 ~&2\\
\end{pmatrix}\hspace{-0.15cm},~
a=\begin{pmatrix}
0  \\
2.2857   \\
 0  \\
\end{pmatrix}\hspace{-0.15cm},~
b=\begin{pmatrix}
4 \\
-15  \\
18  \\
\end{pmatrix}\hspace{-0.15cm},~
\delta =1,~\beta = 4.
\end{align*}
Here we have $\lambda_1 = -8$ and
$\mathrm{dim} ~ \mathrm{Ker}\left( A- \lambda_{min}(A) I_n\right)=1 \ngeq  2$, thus the dimension condition (\ref{eq:rankcond})   fails to hold. Moreover, the new dimension condition of \cite{hs} also fails to hold, since
\begin{align*}
\mathrm{rank}\left( \left[ A-\lambda_1 I_n ~~~  b \right] \right)=
\begin{pmatrix}
4  ~& 0 ~& 0 &4 \\
0  ~& 0 ~& 0 &-15\\
 0 ~& 0 ~& 10 &18\\
\end{pmatrix}=3 \nleqslant 2.
\end{align*}
The global solution of \TRS is $x_g^*=[0, -1, 0]^T$ which is infeasible for \eTRS and \LNGM of \TRS is $x_l^*=[ 0~, 1~,0]^T$ which is feasible for \eTRS with the objective value $-3.4286$. The optimal objective value of SDP relaxation (\ref{sdpr})  is $-11.0642$ and the optimal objective value of  SOCP/SDP formulation (\ref{dsocp}) is $-9.7551$ which is also the optimal objective value of (\ref{etrs1}). The optimal solution of  (\ref{dsocp}) is
$x_{{\tiny socp/sdp}}^*=[-0.2885 ,  -0.8567 ,  -0.4276]^T$
and $X^*=x_{{\tiny socp/sdp}}^*(x_{{\tiny socp/sdp}}^*)^T$, thus $x_{{\tiny socp/sdp}}^*$ is optimal for (\ref{etrs1}).
\end{example}
%%%%%%%%%%%%%%%%%%%%%%%%%%%%%%%%%%%%%%%% EXAMPLE 4
%%%%%%%%%%%%%%%%%%%%%%%%%%%%%%%%%%%%%%%%
%%%%%%%%%%%%%%%%%%%%%%%%%%%%%%%%%%%%%%%%%%
In all three examples above, the optimal solution of (\ref{dsocp}) is rank one, thus we easily have the solution of (\ref{etrs1}). However, this is not the case in general as illustrated by the following example.
\begin{example}\label{Exam3}
Let
\begin{align*}
A=\begin{pmatrix}
-4 ~& 0 ~& 0\\
0  ~& 1 ~& 0\\
 0 ~& 0 ~&-3\\
\end{pmatrix}\hspace{-0.15cm},~
a=\begin{pmatrix}
0.5714  \\
0   \\
 0  \\
\end{pmatrix}\hspace{-0.15cm},~
b=\begin{pmatrix}
-6 \\
-3  \\
0  \\
\end{pmatrix}\hspace{-0.15cm},~
\delta =1,~\beta = 2.2.
\end{align*}
We have $\lambda_1 = -4$ and
$\mathrm{dim} ~ \mathrm{Ker}\left( A- \lambda_{min}(A) I_n\right)=1 \ngeq  2$, thus the dimension condition (\ref{eq:rankcond})   does not hold. Moreover, the new dimension condition \cite{hs} also fails to hold, since
\begin{align*}
\mathrm{rank}\left( \left[ A-\lambda_1 I_n ~~~  b \right] \right)=
\begin{pmatrix}
0  ~& 0 ~& 0 &-6 \\
0  ~& 5 ~& 0 &-3\\
 0 ~& 0 ~& 1 &0\\
\end{pmatrix}=3 \nleqslant 2.
\end{align*}
The global solution of \TRS is $\bar{x}^*=[-1~, 0~, 0]^T$ which is again infeasible for \eTRS and \LNGM of \TRS is $\bar{x}=[  1~,0~,0]^T$ which is feasible for \eTRS with the objective value $-2.8572$. The optimal objective value of SDP relaxation (\ref{sdpr})  is $-5.4354$ and the optimal objective value of  SOCP/SDP formulation (\ref{dsocp}) is $-3.6121$ which is also the optimal objective value of (\ref{etrs1}).
The optimal solution of (\ref{dsocp}) is
\begin{align*}
 X^*=\begin{pmatrix}
0.1842 ~& -0.0537 ~& 0\\
-0.053 ~& 0.0156  ~& 0\\
 0 ~       & 0 ~         &0.8001\\
\end{pmatrix}\hspace{-0.15cm},~~
 x_{{\tiny socp/sdp}}^*=\begin{pmatrix}
-0.4292 \\
0.1251 \\
 0 ~   \\
\end{pmatrix}\hspace{-0.15cm},~
\end{align*}
which obviously $X^*\neq x_{{\tiny socp/sdp}}^*(x_{{\tiny socp/sdp}}^*)^T$.
 By the  rank-one decomposition procedure discussed in the previous section, one gets the optimal solution  $x^*=[-0.4292 ,  0.1251 ,  -0.8945]^T$
 for (\ref{etrs1}). \end{example}

\section{Conclusions}In this paper, using a variant of S-Lemma,  we presented  the necessary and sufficient optimality conditions
for the extended trust region subproblem that led us to an SOCP/SDP reformulation  for it. Our derived formulation turned out to be the dual of the SOCP/SDP formulation given in \cite{ba,sz} but with a completely different approach. Extending this idea for several linear inequality constraints could be an interesting future research direction.

\end{document}